\documentclass[a4paper]{amsart}
\usepackage[utf8]{inputenc}
\usepackage{amssymb,amsmath,enumerate,amsthm,url}
\usepackage[numbers]{natbib}
\usepackage{bibentry}

\theoremstyle{plain}
\newtheorem{theorem}{Theorem}[section]
\newtheorem{lemma}[theorem]{Lemma}
\newtheorem{proposition}[theorem]{Proposition}

\newtheorem{fact}[theorem]{Fact}
\newtheorem*{fact*}{Fact}

\newtheorem{claim}[theorem]{Claim}
\newtheorem*{claim*}{Claim}
\theoremstyle{definition}
\newtheorem{definition}[theorem]{Definition}
\newtheorem*{definition*}{Definition}
\newtheorem{notation}[theorem]{Notation}
\newtheorem*{notation*}{Notation}
\theoremstyle{remark}
\newtheorem{remark}[theorem]{Remark}
\newtheorem*{remark*}{Remark}
\newtheorem{example}[theorem]{Example}
\newtheorem*{example*}{Example}

\newtheorem*{note*}{Note}
\newtheorem{question}[theorem]{Question}
\newtheorem*{question*}{Question}
\begin{document}
\newcommand{\eq}{{\operatorname{eq}}}
\newcommand{\tp}{{\operatorname{tp}}}
\newcommand{\dcl}{{\operatorname{dcl}}}
\newcommand{\acl}{{\operatorname{acl}}}
\newcommand{\im}{{\operatorname{im}}}
\newcommand{\Th}{{\operatorname{Th}}}
\newcommand{\ACVF}{{\operatorname{ACVF}}}
\newcommand{\fin}{{\operatorname{fin}}}
\newcommand{\res}{{\operatorname{res}}}
\newcommand{\alg}{{\operatorname{alg}}}
\newcommand{\lcm}{{\operatorname{lcm}}}
\newcommand{\Gal}{{\operatorname{Gal}}}
\newcommand{\End}{{\operatorname{End}}}
\newcommand{\V}{{\mathbb{V}}}
\newcommand{\N}{{\mathbb{N}}}
\newcommand{\Z}{{\mathbb{Z}}}
\newcommand{\Q}{{\mathbb{Q}}}
\newcommand{\R}{{\mathbb{R}}}
\newcommand{\K}{{\mathbb{K}}}
\newcommand{\F}{{\mathbb{F}}}
\newcommand{\A}{{\mathbb{A}}}
\newcommand{\C}{{\mathbb{C}}}
\newcommand{\U}{{\mathcal{U}}}
\newcommand{\M}{{\mathcal{M}}}
\renewcommand{\b}{{\overline{b}}}
\renewcommand{\d}{{\overline{d}}}
\let\polL\L
\renewcommand{\L}{{\mathcal{L}}}
\renewcommand{\div}{{\operatorname{div}}}

\newcommand{\Lang}{\mathcal{L}}
\newcommand{\Leq}{\mathcal{L}^{\eq}}
\newcommand{\Ldiv}{\mathcal{L}_{\operatorname{div}}}
\newcommand{\Teq}{T^{\eq}}

\newcommand{\Kalg}{K^{\operatorname{alg}}}
\newcommand{\Fpalg}{\F_p^{\operatorname{alg}}}

\newcommand{\defn}[1]{{\bf #1}}

\newcommand{\ns}[1]{{^*}\!#1}
\newcommand{\nsA}{\ns A}

\newcommand{\dcleq}{\dcl^{\eq}}
\newcommand{\acleq}{\acl^{\eq}}

\newcommand{\bigcupdot}{\dot\bigcup}

\newcommand{\eps}{\epsilon}

\newcommand{\TODO}[1]{\em {\bf TODO:} #1}
\newcommand{\FIXME}[1]{\em {\bf FIXME:} #1}

\title{Incidence bounds in positive characteristic via valuations and 
distality}
\author{Martin Bays \& Jean-François Martin}

\begin{abstract}
  We prove distality of quantifier-free relations on valued fields with finite 
  residue field. By a result of Chernikov-Galvin-Starchenko, this yields  
  Szemerédi-Trotter-like incidence bounds for function fields over finite 
  fields. We deduce a version of the Elekes-Szabó theorem for such fields.
\end{abstract}

\maketitle

\section{Introduction}

We obtain the following incidence bound.
\begin{theorem} \label{t:SzTBasic}
  Let $p$ be a prime, and let $K$ be a finitely generated extension of 
  $\F_p$.
  Let $E \subseteq  K^n \times K^m$ be the zero set of a set of polynomials in 
  $K[x_1,\ldots ,x_{n+m}]$.
  Let $d,s \in \N$ and suppose $E$ is $K_{d,s}$-free, i.e.\ if $A\times B \subseteq  E$ 
  then $|A| < d$ or $|B| < s$.

  Then there exists $\epsilon > 0$ (which can in principle be calculated for a 
  given $K$ as a function of the number and degrees of the polynomials 
  defining $E$) and $C > 0$ such that for any finite subsets $A \subseteq  K^n$ and $B 
  \subseteq  K^m$,
  $$|E \cap (A \times B)| \leq  C(|A|^{1-\epsilon}
  |B|^{\frac{d-1}d(1+\epsilon)} + |A| + |B|).$$
\end{theorem}

\subsection{Background and motivation}
The Szemerédi-Trotter theorem bounds the number of point-line incidences 
between a set $P$ of points and a set $L$ of lines in the real plane.
We state a version with an explicit bound, \cite[Theorem~8.3]{TaoVu}:
\begin{fact} \label{f:STOrig}
  For any finite $P$ and $L$,
  $$|\{ (p,l) \in P \times L : p \in l\}|
  \leq  4|P|^{\frac23}|L|^{\frac23}+4|P|+|L|.$$
\end{fact}

Statements of the form of Theorem~\ref{t:SzTBasic} can be seen as generalisations of 
this, replacing the point-line incidence relation with other algebraic binary 
algebraic relations. For characteristic 0 fields, such results were proven 
first in \cite[Theorem~9]{ES}, and subsequently strengthened in
\cite[Theorem~1.2]{FoxEtAl}. Using such bounds for binary relations, 
Elekes-Szabó \cite{ES} obtained analogous bounds for ternary algebraic 
relations.

In positive characteristic, versions of Fact~\ref{f:STOrig} have been proven 
(\cite{BKT},\cite{SZ}) where one restricts to sets which are small compared to 
the characteristic. This is related to the sum-product phenomenon in fields, 
where finite fields are known to be the only obstruction (\cite{BKT}, 
\cite[Theorem~2.55]{TaoVu}).

Meanwhile, a special case of results in \cite{CGS} and 
\cite[Section~2]{CS-ES1d} yields a version of Theorem~\ref{t:SzTBasic} in 
characteristic 0 by seeing it as a consequence of the fact that the complex 
field is a reduct of a \emph{distal} structure, namely the real field.
Distality is a notion originating in model theory which imposes a certain 
combinatorial restriction on the complexity of definable sets in a structure; 
this notion and its incidence theoretic implications are summarised in Section 
\ref{s:CSbackground} below.
It would be surprising if the positive characteristic results mentioned in the 
previous paragraph, which require an unbounded characteristic, could be seen 
as instances of distality. We consider instead the orthogonal situation of a 
function field over a finite field, and we prove Theorem~\ref{t:SzTBasic} by finding 
sufficient distality to trigger the incidence bounds of \cite{CGS}. We obtain 
this distality using elementary notions from the model theory of valued 
fields, and in fact our results apply more generally to any valued field with 
finite residue field. It follows from \cite{KSW} that a positive 
characteristic valued field with finite residue field is not NIP, and so is 
not the reduct of a distal structure; this forces us to use a more local 
notion of distality.

Our motivation for considering these fields is \cite[Section~5]{Hr-psfDims}, 
which suggests a unifying explanation for all the results on existence of 
bounds described above: they are all incarnations of {\em modularity} in the 
model-theoretic sense, and they are consistent with a Zilber dichotomy 
statement of the form ``any failure of modularity arises from an infinite 
pseudofinite field''. In other words, unboundedly large finite fields should 
be the cause of any failure of the bounds. As a special case, this would 
suggest that for a field $K$ of characteristic $p > 0$ which has finite 
\defn{algebraic part} $\K \cap \F_p^\alg$, incidence bounds and Elekes-Szabó 
results should go through as in characteristic 0.

We partially confirm this only in the special\footnote{See 
Proposition~\ref{p:latentFields}} case of fields admitting finite residue field. 
However, in Theorem~\ref{t:ES} we do confirm for such $K$ that an Elekes-Szabó 
result applies: a mild strengthening of Theorem~\ref{t:SzTBasic} suffices as input 
to the proof of one of the main results of \cite{BB-cohMod}, yielding 
Elekes-Szabó bounds for algebraic relations of arbitrary arity and codimension 
in $K^n$ which do not arise from 1-dimensional algebraic groups.

\subsection{Acknowledgements}
Thanks to Artem Chernikov and Sergei Starchenko for conversation which 
launched the project, to Sylvy Anscombe, Philipp Dittmann, Udi Hrushovski, and 
Silvain Rideau-Kikuchi for miscellaneous helpful conversation, and to 
Elisabeth Bouscaren for matchmaking and sanity checking. Thanks also to Martin 
Hils for pointing out an error in the bounds in Remark~\ref{r:bounds} in an earlier 
version of the paper.

\emph{\small{Bays was supported in part by the Deutsche Forschungsgemeinschaft 
(DFG, German Research Foundation) under Germany's Excellence Strategy EXC 
2044–390685587, Mathematics Münster: Dynamics–Geometry–Structure.}}

\section{Preliminaries}

We use basic notions and notation from model theory, as presented in e.g.\ 
\cite[Chapter~1]{TZ}.

Let $\Lang$ be a (possibly many-sorted) first order language and $T$ a 
complete $\Lang$-theory.

\begin{notation}
  If $\M \vDash  T$ and $B \subseteq  \M$ and $x=(x_1,\ldots ,x_n)$ is a tuple of variables of 
  sorts $S_1,\ldots ,S_n$, we write $B^x$ for $\prod_i (S_i(\M) \cap B)$.
  We write $|x|$ for the length $|x| = n$ of the tuple.

  For a set $B$,
  we write $B_0 \subseteq_{\fin}  B$ to mean that $B_0$ is a finite subset of $B$.

  For a formula $\phi$, we define $\phi^0 := \neg\phi$ and $\phi^1 := \phi$.

  If $\phi(x;y)$ is a partitioned formula and $b \in \M^x$ and $A \subseteq  \M$, we 
  set $\tp_\phi(b/A) := \{ \phi(x,c)^\eps : c \in A^y;\; \eps \in \{0,1\};\; \M 
  \vDash  \phi(b,c)^\eps \}$. The partitioning will often be left implicit.
\end{notation}

\section{Distality}
\subsection{Distal cell decompositions}
\label{s:CSbackground}
We recall the following definition from \cite{CGS}:

\begin{definition}
  Let $A$ and $B$ be sets.
  A binary relation $E \subseteq  A \times B$ {\em admits a distal cell decomposition} 
  with exponent $t \in \R$ if there exist $s \in \N$ and finitely many relations 
  $\Delta_i \subseteq  A \times B^s$ and $C \in \R$ such that for every $B_0 \subseteq_{\fin}  B$, 
  $A$ can be written as a (not necessarily disjoint) union of $\leq  C|B_0|^t$ 
  subsets of the form $\Delta_i(c)$ for $c \in {B_0}^s$, each of which is {\em cut} 
  by no $E(b)$ for $b \in B_0$, i.e.\ $\Delta_i(c) \subseteq  E(b)$ or $\Delta_i(c) \cap 
  E(b) = \emptyset $.
\end{definition}

It was proven in \cite{CGS} that relations admitting distal cell 
decompositions enjoy certain incidence bounds. For our purposes, the following 
version of this deduced in \cite[Theorem~2.6,2.7(2)]{CS-ES1d} is most 
relevant.

A binary relation $E \subseteq  A \times B$ is \defn{$K_{d,s}$-free} if it contains no 
subset $A_0 \times B_0$ with $|A_0| = d$ and $|B_0| = s$.

\begin{fact} \label{f:distInc}
  Let $E \subseteq  A \times B$ be $K_{d,s}$-free and admit a distal cell 
  decomposition with exponent $t$. Then for $A_0 \subseteq_{\fin}  A$ and $B_0 \subseteq_{\fin}  B$,
  $$|E \cap (A_0 \times B_0)| \leq  O_E(|A_0|^{\frac{(t-1)d}{td-1}} 
  |B_0|^{\frac{td-t}{td-1}} + |A_0| + |B_0|).$$
\end{fact}

\subsection{Distal subsets}

\begin{definition}
  Let $\M \vDash  T$.
  Let $\phi(x;y)$ be an $\Lang$ formula, and let $A,B \subseteq  \M$ be subsets.
  \begin{itemize}\item An $\Lang$-formula $\zeta_\phi(x;z)$ is a \defn{uniform strong honest 
  definition} (\defn{USHD}) for $\phi$ on $A$ over $B$ if
  for any $a \in A$ and finite subset $B_0 \subseteq_{\fin}  B$ with $|B_0| \geq  2$,
  there is $d \in {B_0}^z$
  such that $\tp(a/B_0) \ni \zeta_\phi(x,d) \vdash  \tp_\phi(a/B_0)$.

  \item We omit ``on $A$'' in the case $A = \M$.

  \item We omit ``over $B$'' in the case $B=A$.

  \item $A$ is \defn{distal in} $\M$ if every $\Lang$-formula $\phi(x;y)$ has a 
  USHD on $A$.
  \end{itemize}
\end{definition}

The notion of a strong honest definition comes from \cite{CS-extDefII}. We 
work with USHDs rather than directly with distal cell decompositions in order 
to be able to reduce to one variable (Lemma~\ref{l:distal1}), and because dealing 
with a single formula is more convenient for many purposes. As the following 
remark makes explicit, there is little difference between the two notions.

\begin{remark} \label{r:distDist}
  An $\Lang$-formula $\phi(x;y)$ has a USHD on $A$ over $B$ if and only if the 
  binary relation $E := \phi(A;B) \subseteq  A^x\times B^y$ admits a distal cell 
  decomposition where the $\Delta_i$ are themselves defined by 
  $\Lang$-formulas. The restriction $|B_0| \geq  2$ allows multiple $\Delta_i$ to 
  be coded as one formula, a trick we will use repeatedly; explicitly, if 
  $\delta_i(x,z_i)$ define $\Delta_i$, then 
  $$\zeta(x,z_1,\ldots ,z_s,w_1,\ldots ,w_s,w'_1,\ldots ,w'_s) :=
  \bigwedge_i (\delta_i(x,z_i) \leftrightarrow  w_i = w'_i)$$
  is a USHD for $\phi$ on $A$ over $B$.

  In particular, if $A \subseteq  \M$ is distal in $\M$, then $\phi(A;A)$ admits a 
  distal cell decomposition for any $\Lang$-formula $\phi(x,y)$.
  It follows that this also holds for any $\Lang(A)$-formula $\phi(x,y)$, 
  since if say $\phi(x,y) = \psi(x,y,a)$ with $\psi(x,y,z)$ an 
  $\Lang$-formula, then a distal cell decomposition $(\Delta_i)_i$ for 
  $\psi(A;A,A) \subseteq  A^x \times (A^y \times A^z)$ yields the distal cell 
  decomposition $(\Delta'_i)_i$ for $\phi(A;A)$, where $(x,(y_1,\ldots ,y_s)) \in 
  \Delta'_i \Leftrightarrow  (x,((y_1,a),\ldots ,(y_s,a)) \in \Delta_i$.\footnote{One can 
  similarly show directly that an $\Lang(A)$-formula has an $\Lang(A)$-formula 
  as a USHD on $A$. The published version of this paper incorrectly formulated 
  the definition of ``distal in'' to allow $\phi$ to be an $\Lang(A)$-formula 
  while, if read literally, still requiring the USHD to be an $\Lang$-formula. 
  Thanks to Mira Tartarotti for pointing out this error.}

\end{remark}

\subsection{Reductions}

\begin{lemma} \label{l:distal1}
  A subset $A \subseteq  \M$ is distal in $\M$ if and only if any $\Lang$-formula 
  $\phi(x;y)$ with $|x| = 1$ has a USHD on $A$.
\end{lemma}
\begin{proof}
  It follows by an inductive argument from the 1-variable case that any 
  $\Lang$-formula has a USHD on $A$; we refer to the proof of
  \cite[Proposition~1.9]{ACGZ} for this argument.
\end{proof}

%
%

\begin{lemma} \label{l:dd-feq}
  Let $\M$ be an $\Lang$-structure.
  Let $S$ and $\widetilde {S}$ be $\Lang$-sorts and let $f : \widetilde {S} \rightarrow  S$ be an 
  $\Lang$-definable function with uniformly finite fibres,
  say $|f^{-1}(b)| \leq  N$ for all $b \in f(\widetilde {S}(\M))$.
  Suppose $B \subseteq  f(\widetilde {S}(\M))$, and let $\widetilde {B} := f^{-1}(B) \subseteq  \widetilde {S}(\M)$.

  Let $A \subseteq  \M^x$ and
  let $\phi(x,y)$ be an $\Lang$-formula
  such that $\phi(x;f(z))$ has a USHD on $A$ over $\widetilde {B}$.
  Then $\phi(x;y)$ has a USHD on $A$ over $B$.
\end{lemma}
\begin{proof}
  Say $\zeta(x,w)$ is a USHD for $\phi(x;f(z))$ over $\widetilde {B}$.

  Let $B_0 \subseteq_{\fin}  B$ and $a \in A$. Then $\widetilde {B}_0 := f^{-1}(B_0)$ is a finite subset 
  of $\widetilde {B}$, so there is $\widetilde {d}$ such that $\tp(a/\widetilde {B}_0) \ni \zeta(x,\widetilde {d}) \vdash  
  \tp_{\phi(x;f(z))}(a/\widetilde {B}_0) \vdash  \tp_{\phi(x;y)}(a/B_0)$.

  \newcommand{\epsbar}{\overline{\eps}}
  Let $d := f(\widetilde {d})$. Then $|f^{-1}(d)| \leq  N^{|w|}$, and so there is $l < 
  N^{|w|}$ and $\epsbar_0 \in \{0,1\}^l$ and $\b_0 \in (B_0)^l$ such that
  $\theta_{l,\epsbar_0}(w,d,\b_0)$ has the minimal number of realisations 
  amongst the formulas
  $$\theta_{n,\epsbar}(w,d,\b) := (f(w) = d \wedge \forall x. (\zeta(x,w) \rightarrow  
  \bigwedge_{i=1}^n \phi(x,b_i)^{\eps_i}))$$
  which hold of $\widetilde {d}$, with $n \in \N$ and $\epsbar \in \{0,1\}^n$ and $\b \in 
  (B_0)^n$. The bound $l < N^{|w|}$ follows from the observation that if such 
  a formula does not have the minimal number of realisations, then a single 
  new instance of $\phi$ can be added to reduce the number of realisations.
  By the minimality, we have for any $b \in B_0$ that 
  $\theta_{l,\epsbar_0}(w,d,\b_0) \vdash  \forall x. (\zeta(x,w) \rightarrow  
  \phi(x,b)^\eps)$ for some $\eps \in \{0,1\}$.

  So $\tp(a/B_0) \ni \exists w. (\theta_{l,\epsbar_0}(w,d,\b_0) \wedge \zeta(x,w)) 
  \vdash  \tp_{\phi(x;y)}(a/B_0)$. Coding the finitely many such formulas with $l < 
  N^{|w|}$ and $\epsbar_0 \in \{0,1\}^l$ into a single formula,
  we therefore obtain a USHD for $\phi(x;y)$ on $A$ over $B$.

\end{proof}

\begin{remark}
  The finiteness assumption in Lemma~\ref{l:dd-feq} is necessary.
  Consider for example the structure $(X,O_X;<)$ where $X$ is a set, $O_X$ is 
  the set of linear orders on $X$, and $x<_ox'$ is the corresponding ternary 
  relation. Let $\pi_1 : X \times O_X \rightarrow  X$ be the projection.
  As one may see by considering automorphisms, the induced structure on $X$ is 
  trivial, so $x = y$ has no USHD on $X$ over $X$.
  But $x = \pi_1(z)$ has a USHD on $X$ over $X\times O_X$ (since if $X_0 \subseteq_{\fin}  
  X$ and $o \in O_X$, then $\tp_=(x/X_0)$ is implied by the $<_o$-cut of $x$ in 
  $X_0$).
\end{remark}


\subsection{Remarks}
We add some further remarks concerning these definitions, which will not be 
used subsequently.

\begin{remark} \label{r:distalElem}
  Suppose $A$ is distal in an $\Lang$-structure $\M$.
  Then this is expressed in the $\Lang_P$-theory of $(\M;A)$, where $P$ is a 
  new predicate interpreted as $A$;
  i.e.\ if $(\M';A') \equiv  (\M;A)$,
  then $A'$ is distal in $\M'$.
\end{remark}

\begin{remark}
  By \cite[Theorem~21]{CS-extDefII}, $\Th(\M)$ is distal if and only if $\M$ 
  is distal in $\M$.
  (No saturation assumption is needed here, thanks to
  Remark~\ref{r:distalElem}.)
\end{remark}

\begin{remark}
  Distality in $\M$ of a subset $A \subseteq  \M$ is equivalent to distality of the 
  induced structure $(A;(\phi(A)_{\phi \text{ an $\L$-formula}}))$ if this 
  structure admits quantifier elimination, but in general is much weaker. We 
  could say that distality of a subset means that it has ``quantifier-free 
  distal induced structure''.
\end{remark}

\begin{example} \label{e:indSeq}
  If $A = (a_i)_{i \in I} \subseteq  \M$ is an $\emptyset $-indiscernible sequence which is not 
  totally indiscernible, and this is witnessed by an $\Lang$-formula 
  $\theta_<$ with $\M \vDash  \theta_<(a_i,a_j) \Leftrightarrow  i<j$,
  then $A$ is distal in $\M$.
\end{example}

\begin{remark}
  The argument of \cite{CS-extDefII} to obtain uniformity of honest 
  definitions goes through in this setting.
  Namely, if $A$ is a subset of a model $\M$ of a complete NIP $\Lang$-theory 
  $T$,
  and the $\Lang_P$-structure $(\M;A)$ is $|T|^+$-saturated,
  then $A$ is distal in $\M$ if and only if
  for any singleton $a \in A$ and any subset $B \subseteq  A$, $\tp^{\M}(a/B)$ is 
  compressible in the sense of \cite{simon-decomposition}. This follows from a 
  ``$(p,q)$-argument'' and transitivity of compressibility.


  It follows in particular that Example~\ref{e:indSeq} can be generalised slightly 
  when $\M$ is NIP: any $\emptyset $-indiscernible sequence which is not totally 
  indiscernible is distal in $\M$.
\end{remark}

\begin{question}
  The following question was asked by Hrushovski and Pillay. By a result of 
  Simon, an NIP theory is distal if and only if every generically stable 
  Keisler measure is smooth. Does a version of this result go through for 
  distality of subsets of NIP structures? Is $A$ distal in $\M$ if and only if  
  every generically stable Keisler measure on $\Th_{\Lang_P}(\M,A)$ with 
  $\mu(\neg P)=0$ is smooth?
  This might provide an alternative route to Theorem~\ref{t:fin-res-distal}.
\end{question}

\section{Fields admitting valuations with finite residue field}
By classical results in valuation theory, a valuation on a field $K$ can be 
extended to any finite extension of $K$ with a finite extension of the residue 
field \cite[Theorem~3.1.2, Corollary~3.2.3]{EnglerPrestel}, and can be 
extended to the transcendental extension $K(X)$ without extending the residue 
field \cite[Corollary~2.2.3]{EnglerPrestel}. Since $\F_p$ and $\Q$ admit 
valuations with finite residue field (respectively trivial and $p$-adic), we 
inductively obtain:\

\begin{lemma} \label{l:fgRes}
  Let $K$ be a finitely generated field.
  Then $K$ admits a valuation with finite residue field.
\end{lemma}

If $K$ is a valued field of characteristic $p>0$, then the induced valuation 
on the algebraic part $K\cap \F_p^\alg$ is trivial. So a positive 
characteristic field which admits a valuation with finite residue field has 
finite algebraic part.
However, the converse fails.

\begin{proposition} \label{p:latentFields}
  For any prime $p$,
  there exists an algebraic extension $L \geq  \F_p(t)$ such that $L\cap 
  \F_p^\alg = \F_p$ but no valuation on $L$ has finite residue field.
\end{proposition}
\begin{proof}
  We work in an algebraic closure $\F_p(t)^\alg$ of $\F_p(t)$.
  Let $\wp : \F_p(t)^{\alg} \rightarrow  \F_p(t)^\alg$ be the Artin-Schreier map 
  $\wp(x) := x^p-x$, an additive homomorphism with kernel $\F_p$.

  \begin{claim*}
    $\deg(\F_p(t,(\wp^{-1}(t^a))_{a > 0}) / \F_p(t))$ is infinite.
  \end{claim*}
  \begin{proof}
    By \cite[Theorem~8.3]{Lang-algebra},
    it suffices to see that $\{ t^a | a > 0\}$ is not contained in any finite 
    union of additive cosets of $\wp(\F_p(t))$.
    Let $a,b \in \N \setminus p\N$ be distinct. Let $\beta_{a,b} := \sum_{i \geq  0} 
    (t^{ap^i} - t^{bp^i}) \in \F_p[[t]]$. Then $\wp(\beta_{a,b}) = t^b - t^a$.
    Now $\beta_{a,b} \notin \F_p(t)$, since there are arbitrarily long intervals 
    between exponents with non-zero coefficient in this power series.
    So $(t^a)_{a \in \N \setminus p\N}$ lie in distinct cosets of $\wp(\F_p(t))$.
  \end{proof}

  We write $\res$ for the residue field map associated to a chosen valuation 
  $v$ on a field $K$, and $\res(K)$ for the corresponding residue field.

  \begin{claim*}
    Let $K' \geq  K \geq  \F_p(t)$ be finite field extensions, and suppose $K' \cap 
    \F_p^\alg = \F_p$.
    Let $v$ be a valuation on $K$ with $\res(K)$ finite.

    Then there exists a finite field extension $K'' \geq  K'$ such that
    $K'' \cap \F_p^\alg = \F_p$
    but for any extension of $v$ to $K''$, $\res(K'') \gneq  \res(K)$.
  \end{claim*}
  \begin{proof}
    The valuation $v$ is non-trivial, so say $v(s) > 0$.
    So $v$ induces the $s$-adic valuation on $\F_p(s) \leq  K$.
    Now $s$ is transcendental, so $t$ is algebraic over $s$, so $K$ is also a 
    finite extension of $\F_p(s)$.
    So we may assume without loss that $v$ restricts to the $t$-adic valuation 
    on $\F_p(t)$.

    Since $\res(K)$ is finite, it is not Artin-Schreier closed;
    say $\alpha \in \res(K) \setminus \wp(\res(K))$.
    Say $\res(\bar\alpha) = \alpha$.

    Since $\deg(K'/\F_p(t))$ is finite, it follows from the above Claim that 
    $$\deg(K'(\wp^{-1}(\bar\alpha),(\wp^{-1}(\bar\alpha + t^a))_{a > 0})/K')$$ 
    is infinite.
    So say $a>0$ is such that $K'' := K(\wp^{-1}(\bar\alpha+t^a)) \not\subseteq  
    K'(\F_{p^p})$.
    Then by considering degrees, $K'' \cap \F_p^\alg = \F_p$.
    But for any extension of $v$ to $K''$,
    $$\wp(\res(\wp^{-1}(\bar\alpha+t^a))) = \res(\bar\alpha+t^a) = \alpha \notin 
    \wp(\res(K)),$$
    so $\res(K'') \gneq  \res(K)$.
  \end{proof}

  Now we recursively construct a chain $K_0 := \F_p(t) \leq  K_1 \leq  \ldots $
  of finite extensions of $\F_p(t)$.
  Let $\eta : \omega \times \omega \rightarrow  \omega$ be a bijection such that 
  $\eta(i,j) \geq  i$ for all $i,j$.

  Note that $\F_p(t)$ admits only countably many valuations (identifying a 
  valuation with its valuation ring); indeed, as above, each non-trivial 
  valuation is a finite extension of the $s$-adic valuation on some $\F_p(s) 
  \leq  \F_p(t)$; there are only countably many choices for $s$, and only 
  finitely many ways to extend a valuation to a finite extension 
  (\cite[Theorem~3.2.9]{EnglerPrestel}).
  Hence also there are also only countably many valuations on each $K_i$.
  Once $K_i$ is constructed, let $\{v_{i,j} : j \in \omega\}$ be the set of 
  valuations on $K_i$ with finite residue field.

  Suppose $k=\eta(i,j)$ and $K_k$ has been constructed. Let $K_{k+1} \geq  K_k$ 
  be an extension as in the second Claim for the extensions $K_k \geq  K_i \geq  
  \F_p(t)$ and the valuation $v_{i,j}$ on $K_i$.

  Now let $K_\omega := \bigcup_{k < \omega} K_k$.
  We have $K_\omega \cap \F_p^\alg = \F_p$ since this holds for each $K_k$.

  Suppose $v$ is a valuation on $K_\omega$ with finite residue field.
  Then $\res(K_\omega) = \res(K_i)$ say, and the restriction of $v$ to $K_i$ 
  is $v_{i,j}$ say.
  Then $\res(K_{\eta(i,j)+1}) = \res(K_i)$, contradicting the construction.
\end{proof}

\begin{remark}
  One might expect that a Zorn argument could replace the recursive 
  construction of the previous Proposition, i.e.\ that any maximal regular 
  extension of $\F_p(t)$ has no valuation with finite residue field. But 
  $\F_p((t^\Q)) \cap \F_p(t)^\alg$ is a counterexample. Thanks to Zoé 
  Chatzidakis for pointing this out.
\end{remark}

\section{Distality in $\ACVF$ of subfields with finite residue field}
\subsection{Uniform Swiss cheese decompositions}
Let $L$ be a non-trivially valued algebraically closed field.
Write $v$ for the valuation map and $\res$ for the residue field map.
We consider $L$ as an $\Ldiv := \{+,-,\cdot,|,0,1\}$-structure, where $x|y \Leftrightarrow  
v(x) \leq  v(y)$; by a result of Robinson, $L$ has quantifier elimination in this 
language.
An open resp.\ closed \defn{ball} in $L$ is a definable set of the form $\{ x : 
v(x-a) > \alpha \}$ resp.\ $\{ x : v(x-a) \geq  \alpha \}$, where $a \in L$ and 
$\alpha \in v(L) \cup \{-\infty,+\infty\}$.

\begin{fact}[Canonical Swiss cheese decomposition {\cite[Theorem~3.26]{Holly}}] \label{f:swiss}
  Any boolean combination of balls can be represented as a finite disjoint 
  union of ``Swiss cheeses'' $\bigcupdot_{i<k} (b_i \setminus \bigcupdot_{j<k_i} 
  b_{ij})$, where the $b_i$ are balls, each $b_{ij}$ is a proper sub-ball of 
  $b_i$, for each $i$ the $b_{ij}$ are disjoint, and no $b_i$ is equal to any 
  $b_{i'j}$. This representation is unique up to permutations.
\end{fact}

We call the $b_i$ the ``rounds'' and the $b_{ij}$ the ``holes'' of a Swiss 
cheese decomposition, and we say such a decomposition has \defn{complexity} 
$\leq N$ if there are $k\leq N$ rounds each with $k_i\leq N$ holes.
Let $\phi(x,y)$ be an $\Ldiv$-formula with $|x| = 1$.
For any $a \in L$, it follows directly from quantifier elimination that 
$\phi(L,a)$ is a boolean combination of balls. We will need the following form 
of uniformity in $a$ of the Swiss cheese decompositions.

\begin{lemma} \label{l:ACVFQE-uniformity}
  There are $N$ and $d$ depending only on $\phi$ such that for all $a \in L^y$,
  $\phi(L,a)$ has a Swiss cheese decomposition of complexity $\leq N$, each round 
  and each hole of which contains a point in a field extension of the subfield 
  generated by $a$ of degree dividing $d$.
\end{lemma}
\begin{proof}
  By quantifier elimination,
  $\phi(x,a)$ is equivalent to a boolean combination of formulas of the form
  $\phi_i(x,a) := v(f_i(x,a)) < v(g_i(x,a))$
  for polynomials $f_i,g_i \in \Z[x,y]$.

  Now if $f_i(x,a)$ is non-constant, then $v(f_i(x,a)) = \sum_{j=1}^k 
  v(x-\alpha_j)$ where $\alpha_1,\ldots ,\alpha_k \in L$ are the zeroes of 
  $f_i(x,a)$. Similarly for the $g_i$. So the following claim shows that 
  $\phi(x,a)$ is equivalent to a boolean combination of balls centred at the 
  roots of the non-constant $f_i(x,a)$ and $g_i(x,a)$.

  \begin{claim}
    Let $s \in \N$, $\gamma_1,\ldots ,\gamma_s \in L$, $n_1,\ldots ,n_s \in \Z$, and $\nu 
    \in v(L)$.
    Then the affine linear constraint $\psi(x) := \sum_{i=1}^s n_i 
    v(x-\gamma_i) < \nu$ is equivalent to a boolean combination of balls 
    centred at the $\gamma_i$.
  \end{claim}
  \begin{proof}
    We prove this by induction on $s$, the case $s=0$ being trivial.

    There are only finitely many possible order types for
    $\eps := v(x-\gamma_1)$ over $\{ v(\gamma_1-\gamma_i) : i > 1 \}$.
    Fix one such order type, and let $X$ be the set of $x$ for which $\eps$ 
    has this order type. Since $X$ is itself a boolean combination of balls 
    centred at $\gamma_1$, it suffices to show that $\phi(x)$ is equivalent on 
    $X$ to a boolean combination of balls centred at the $\gamma_i$.

    If $\eps = v(\gamma_1-\gamma_i)$ for some $i>1$, then $\psi(x)$ is 
    equivalent on $X$ to $\sum_{i=2}^s n_i v(x-\gamma_i) < \nu - 
    n_1v(\gamma_1-\gamma_i)$, and we conclude by the inductive hypothesis. 

    Otherwise, by the ultrametric triangle inequality, $v(x-\gamma_i) = \eps$ 
    if $\eps < v(\gamma_1-\gamma_i)$, and $v(x-\gamma_i) = 
    v(\gamma_1-\gamma_i)$ otherwise, so $\psi(x)$ is equivalent on $X$ to 
    $nv(x-\gamma_1) < \nu'$ or $nv(x-\gamma_1) > \nu'$ for some $n \in \Z$ and 
    $\nu' \in v(L)$.
  \end{proof}

  So $\phi(L,a)$ is a boolean combination of balls each having a point in an 
  extension of the subfield generated by $a$ of degree dividing $$d := \lcm_i 
  (\lcm(\deg_x f_i, \deg_x g_i)).$$
  Refining this boolean combination to the Swiss cheese decomposition, it 
  follows that its rounds and holes also have this property. Indeed, since no 
  ball is the union of finitely many subballs, the rounds and holes must 
  appear in any expression for $\phi(L,a)$ as a boolean combination of balls.

  We may assume $L$ is $\aleph_0$-saturated, and so by compactness we obtain a 
  uniform bound on the number of rounds and holes, as required.
\end{proof}

\subsection{Compressing cheeses}
Let $B$ be the imaginary sort of $L$ consisting of balls, both open and 
closed, including the empty ball and its complement. We write $x \in b$ for the 
corresponding $\emptyset $-definable (in $\Leq$) membership relation $( \in ) \subseteq  L 
\times B$.

Given $N \in \N$, let $S_N$ be the imaginary sort of $L$ which codes Swiss 
cheese decompositions of complexity at most $N$. This means that we have an 
associated $\emptyset $-definable membership relation, which we also write as $( \in ) 
\subseteq  L \times S_N$, such that $c_1 = c_2$ iff $\{ x : x \in c_1 \} = \{ x : x \in 
c_2 \}$, and
setting $$X_N := \{(b_i)_{i<N},(b_{ij})_{i,j<N} : b_i,b_{ij} \in B \textrm{ are 
as in Fact~\ref{f:swiss}}\} \subseteq  B^{N(N+1)}$$
we obtain a $\emptyset $-definable surjection $f_{S_N} : X_N \twoheadrightarrow  S_N$ defined by 
$f_{S_N}((b_i)_i,(b_{ij})_{ij}) := [{ \textrm{code of } \bigcup_i (b_i \setminus 
\bigcup_j b_{ij}) }]$.
By Fact~\ref{f:swiss}, any $c \in S_N$ has a unique-up-to-permutation representation 
as a Swiss cheese decomposition of complexity $\leq N$, so $f_{S_N}$ has finite 
fibres.

With a view to proving Theorem~\ref{t:fin-res-distal}, for $K$ a valued subfield of 
$L$ with finite residue field and $d \in \N$, define $B_{K,d} \subseteq  B$ to be the 
set of balls which contain an element of some finite field extension of $K$ 
within $L$ of degree dividing $d$ over $K$.
Let $X_N(B_{K,d}) := X_N \cap {B_{K,d}}^{N(N+1)}$
and $S_N(B_{K,d}) := f_{S_N}(X_N(B_{K,d}))$.

\begin{lemma} \label{l:B-rel-distal-local}
  Let $N,d \in \N$.
  \begin{enumerate}[(i)]\item $x \in y$ has a USHD over $B_{K,d}$.
  \item $x \in f_{S_N}(y)$ has a USHD over $B_{K,d}$.
  \item $x \in z$ has a USHD over $S_N(B_{K,d})$.
  \end{enumerate}
\end{lemma}
\begin{proof}
  \begin{enumerate}[(i)]\item
  By assumption, the residue field of $K$ is a finite field, say $\F_q$.

  Let $B_0 \subseteq_{\fin}  B_{K,d}$.
  Let $B_0' := \{ b \vee b' : b,b' \in B_0 \}$ where the join $b \vee b'$ is the 
  smallest ball containing both $b$ and $b'$. By the ultrametric triangle 
  inequality, $B_0'$ is then closed under join. Note that $B_0' \subseteq  B_{K,d}$, 
  since $B_{K,d}$ is upwards-closed.

  Let $p \in L$.
  Let $b \in B_0' \cup \{L\}$ be minimal such that $p \in b$,
  and let $b_1,\ldots ,b_s \in B_0'$ be the maximal proper subballs (if any) of $b$ 
  in $B_0'$.
  Then
  $$(x \in b \wedge \bigwedge_{i=1}^s x \notin b_i) \vdash  \tp_{x \in y}(p/B_0),$$
  and each $b_i$ is the join of two balls in $B_0$, and either the same goes 
  for $b$ or $b=L$. So coding the finitely many possibilities yields a USHD as 
  required if we can bound $s$ independently of $p$.

  Assume $s > 1$.
  Say $p_i \in b_i$ is of degree dividing $d$ over $K$, and let $\alpha \in 
  v(L)$ be the valuative radius of $b$.
  Then $v(p_i-p_j)=\alpha$ for $i\neq j$, since $b_i \vee b_j = b$ (in particular, 
  $b \neq  L$).

  Then $i \mapsto  \lambda_i := \res(\frac{p_i-p_1}{p_2-p_1})$ is an injection of 
  $\{1,\ldots ,s\}$ into $\res(L)$.
  Indeed, if $\lambda_i=\lambda_j$ then $\res(\frac{p_i-p_j}{p_2-p_1}) = 0$, 
  so $v(p_i-p_j) > v(p_2-p_1) = \alpha$, so $i=j$.

  Since each $\lambda_i$ is in the residue field of an extension of $K$ of 
  degree dividing $d^3$,
  by the valuation inequality
  (\cite[Corollary~3.2.3]{EnglerPrestel}).
  each $\lambda_i$ is in an extension of $\res(K) = \F_q$ of degree $\leq  d^3$,
  so $\lambda_i \in \F_{q^{d^3!}}$ for all $i$.
  So $s \leq  q^{d^3!}$.

  \item
  $x \in f_{S_N}(y)$ is equivalent, by the definition of $f_{S_N}$, to a 
  certain boolean combination of the formulas $(x \in y_i)_{i < N(N+1)}$.
  So by (i), coding these formulas yields a formula which is a USHD for $x \in 
  f_{S_N}(y)$ over $B_{K,d}$.

  \item
  Considering now $X_N$ as a sort and $y$ as a variable of sort $X_N$,
  it follows from (ii) that $x \in f_{S_N}(y)$ has a USHD over $X_N(B_{K,d})$.
  Then we conclude by Lemma~\ref{l:dd-feq}.
  \end{enumerate}
\end{proof}

\subsection{Concluding distality}
\begin{lemma} \label{l:fin-res-comp}
  Let $L$ be a non-trivially valued algebraically closed field.
  Let $K \leq  L$ be a subfield and suppose $\res(K)$ is finite.
  Let $\phi(x;y)$ be an $\Ldiv$-formula with $|x| = 1$.
  Then $\phi$ has a USHD over $K$.

  Moreover, for any $r \geq  1$, $\phi$ has a USHD over the set $K_r \subseteq  K^{\alg} 
  \subseteq  L$ of points with degree over $K$ dividing $r$,
  $$K_r := \{ a \in L : \deg(K(a)/K) | r \}$$
\end{lemma}
\begin{proof}
  Let $N$ and $d$ be as in Lemma~\ref{l:ACVFQE-uniformity} for $\phi$.
  Then there is a $\emptyset $-definable function $h : L^y \rightarrow  S_N$ such that
  $L^{\eq} \vDash  \forall x,y. (\phi(x,y) \leftrightarrow  x \in h(y))$,
  and $h(K_r) \subseteq  S_N(B_{K,dr})$.

  By Lemma~\ref{l:B-rel-distal-local}(iii),
  say $\zeta(x,z'_1,\ldots ,z'_s)$ is a USHD for $x \in y$ over $S_N(B_{K,dr})$.
  Then (an $\Ldiv$-formula equivalent to) $\zeta(x,h(z_1),\ldots ,h(z_s))$ is a 
  USHD for $\phi(x;y)$ over $K_r$.
\end{proof}

\begin{theorem} \label{t:fin-res-distal}
  Let $K$ be a valued field with finite residue field.

  Let $L \geq  K$ be an algebraically closed valued field extension.

  Then $K$ is distal in $L$,
  as is each $K_r$ defined as in Lemma~\ref{l:fin-res-comp}.
\end{theorem}
\begin{proof}
  We may assume that $L$ is non-trivially valued, as otherwise $K$ is finite 
  and the result is trivial.

  The result then follows from Lemma~\ref{l:fin-res-comp} and Lemma~\ref{l:distal1}.
\end{proof}

\begin{remark}
  This does not reprove distality of $\Q_p$, because $\Q_p$ does not eliminate 
  quantifiers in $\Ldiv$.
\end{remark}

\section{Incidence theory consequences}

\begin{theorem} \label{t:SzT}
  Let $K$ be a valued field with finite residue field.
  Let $E \subseteq  K^n \times K^m$ be quantifier-free definable in $\Ldiv(K)$.
  Suppose $E$ omits $K_{d,s}$, where $d,s \in \N$.
  Then there exist $t$ (see Remark~\ref{r:bounds}) and $C>0$ such that
  for $A_0 \subseteq_{\fin}  K^n$ and $B_0 \subseteq_{\fin}  K^m$,
  $$|E \cap (A_0 \times B_0)| \leq  C(|A_0|^{\frac{(t-1)d}{td-1}} 
  |B_0|^{\frac{td-t}{td-1}} + |A_0| + |B_0|).$$

  The same holds if $K$ is replaced by $K_r \subseteq  K^{\alg}$ defined as in 
  Lemma~\ref{l:fin-res-comp}.
\end{theorem}
\begin{proof}
  By Theorem~\ref{t:fin-res-distal} and Remark~\ref{r:distDist},
  $E$ admits a distal cell decomposition, and we conclude by Fact~\ref{f:distInc}.
\end{proof}

The version of this stated in the introduction, Theorem~\ref{t:SzTBasic}, follows by 
considering Lemma~\ref{l:fgRes} and the special case that $E$ is defined as the zero 
set of polynomials over $K$, and setting $\epsilon := \frac1{dt-1}$.

\begin{remark} \label{r:bounds}
  By examining the proof, in the case $n=1$ one can obtain a bound on the 
  exponent of the resulting distal cell decomposition giving $t \leq  
  2(q^{d^3!}+1)$ where $q = |\res(K)|^r$ and $d$ is as in the proof of 
  Lemma~\ref{l:fin-res-comp}. Indeed, this is the exponent arising from bounding the 
  number of balls used in Lemma~\ref{l:B-rel-distal-local}(i), and neither 
  Lemma~\ref{l:B-rel-distal-local}(ii) nor Lemma~\ref{l:dd-feq} increase the exponent (for 
  the latter case, this follows from the structure of the proof, since each 
  instance $\zeta(x,\widetilde {d})$ gives rise to a single instance of the eventual 
  formula).

  So we obtain the corresponding explicit bounds in Fact~\ref{f:distInc}. However, 
  we have no reason to expect these bounds to be anything like optimal.

  For $n>1$, calculating explicit bounds is complicated by the fact that when 
  reducing to one variable a USHD for a quantified formula is used, so one 
  needs a bound on the degrees in the quantifier-free formula obtained by 
  quantifier elimination in ACVF. This quantifier elimination is primitive 
  recursive \cite{Weispfenning}, so in principle this could be done, yielding 
  an effective algorithm for computing an exponent $t$ for a given $E$ 
  (uniform in definable families). But we do not attempt to make this explicit 
  here.

  Instead, we illustrate the idea by showing that in the special case of 
  Szemerédi-Trotter, $E = \{ ((x,y),(a,b)) : y=ax+b \}$, we can take $t := 
  4(q+1)$.

\newcommand{\y}{{\overline{y}}}
\newcommand{\z}{{\overline{z}}}
\newcommand{\w}{{\overline{w}}}
  The proof of Lemma~\ref{l:B-rel-distal-local}(i) in this case gives a USHD 
  $\zeta(y,x,\z)$ for $\phi(y;x,(a,b)) := (x,y)E(a,b)$ over $K$, expressing 
  that $y$ is an element of a boolean combination of the points 
  $z_{i,1}x+z_{i,2}$ and the balls spanned by pairs of such points, with at 
  most $2(q+1)$ such points involved. Using coding to choose the form of the 
  boolean combination, this has exponent $\leq 2(q+1)$.

  By \cite[Theorem~2.1]{Weispfenning}, if an $\Ldiv$ qf-formula 
  $\psi(x,\y,\z)$ is linear in $x,\y$, i.e.\ each polynomial has degree 1 in 
  $x$ and each $y_i$, then $\exists x. \psi(x,\y,\z)$ is equivalent modulo 
  $\ACVF$ to a qf-formula linear in $\y$.

  Now the formula $\zeta(y,x,\z) \rightarrow  (x,y)Ew$ is linear in $x,y$,
  so $\forall y. (\zeta(y,x,\z) \rightarrow  (x,y)Ew)$ is equivalent to a qf-formula 
  which is linear in $x$.
  Similarly $\forall y. (\zeta(y,x,\z) \rightarrow  \neg(x,y)Ew)$ is equivalent to a 
  qf-formula linear in $x$, and the two can be coded into a single qf-formula 
  linear in $x$. This then itself admits (by the $n=1$ case of the present 
  Remark with $d=1$) a USHD $\xi(x,\w)$ over $K$ of exponent $\leq 2(q+1)$.
  Then $\xi(x,\w) \wedge \zeta(y,x,\z)$ is a USHD for $E$ over $K$ of exponent 
  $\leq 2(q+1) + 2(q+1) = 4(q+1)$.
  In symmetric form, this gives a bound of $O(N^{\frac32 - 
  \frac1{16(q+1)-2}})$ on the number of incidences of $N$ lines and $N$ 
  points.

  As a final remark, note that although Theorem~\ref{t:SzT} does apply in 
  characteristic 0, e.g.\ to $K=\Q_p$, the bounds we obtain in this way are 
  worse than those obtained from \cite{FoxEtAl} by embedding $\Q_p$ in 
  $\C=\R^2$, even for $p=2$.
\end{remark}

\begin{question}
  Is the dependence on $q$ in these bounds necessary?
  For example, does there exist $\epsilon>0$ such that for all primes $p$ 
  there exists $C$ such that for all $X,A \subseteq  \F_p(t)^2$ we have 
  $|\{((x,y),(a,b)) \in X\times A : y=ax+b\}| \leq  
  C\max(|X|,|A|)^{\frac32-\epsilon}$?
  (Remark~\ref{r:bounds} yields a bound depending on $p$ of $\epsilon = 
  \frac1{16p+14}$ in this case. Meanwhile one can obtain a lower bound 
  exponent of $\frac43$ by considering a rectangular example with bounded 
  degree polynomials, $\F_p[t]_{<n} \times \F_p[t]_{<2n}$.)
\end{question}

\section{Elekes-Szabó consequences}
\newcommand{\ccl}{{\operatorname{ccl}}}
\newcommand{\trd}{{\operatorname{trd}}}
\providecommand{\bdl}{\boldsymbol\delta}
Elekes-Szabó \cite{ES} exploit incidence bounds in characteristic zero to find 
that commutative algebraic groups are responsible for ternary algebraic 
relations with asymptotically large intersections with finite grids.
In \cite{BB-cohMod}, this is generalised to relations of arbitrary arity. In 
this section, we remark that these arguments go through in the present 
positive characteristic context, at least if we restrict to the 1-dimensional 
situation of \cite[Theorem~1.4]{BB-cohMod}.

Let $K_0$ be a field admitting a valuation with finite residue field (e.g.\ a 
function field over a finite field).
Let $\U$ be a non-principal ultrafilter on $\omega$.
Define
$$K' := ((K_0)^\U)^\alg \leq  ((K_0)^\alg)^\U =: K.$$
For $r\geq 1$, let
$$K_r := \{ a \in (K_0)^\alg : \deg(K_0(a)/K_0) | r  \}.$$
\newcommand{\Los}{\polL{}o\'s}
So (by \Los's theorem) we have $K' = \bigcup_{r \in \omega} (K_r)^\U$.

We work with the setup of \cite[2.1]{BB-cohMod}, with $(K_0)^\alg$ in place of 
$\C$, and in a countable language in which each of these internal sets 
$(K_r)^\U \subseteq  K$ is definable. 

\begin{theorem} \label{t:ES}
  Let $K_0$ be a field admitting a valuation with finite residue field.
  Let $V \subseteq  \A^n$ be an affine algebraic variety defined over $K_0$ of 
  dimension $d$.
  Then at least one of the following holds:
  \begin{enumerate}[(i)]\item $V$ admits a powersaving on $K_0$: there exist $C,\eps > 0$ such that 
  for all $X_i \subseteq_{\fin}  K_0$, $i=1,\ldots ,n$, we have
  $$|V(K_0) \cap \prod_i X_i| \leq  C(\max_i |X_i|)^{d-\eps}.$$
  \item $V$ is \emph{special}: $V$ is in co-ordinatewise 
  correspondence\footnote{As defined in \cite[Definition~1.1]{BB-cohMod}} with 
  a product $\prod_i H_i \leq  \prod {G_i}^{n_i}$ of connected subgroups of $H_i$ 
  of powers $G_i$ of 1-dimensional algebraic groups.
  \end{enumerate}
\end{theorem}
\begin{proof}
\renewcommand{\a}{{\overline{a}}}
  Let $K' \leq  K$ be as above. Also let $C_0 \leq  K'$ be a countable algebraically 
  closed subfield over which $V$ is defined.

  The proof in \cite{BB-cohMod} goes through, but using Theorem~\ref{t:SzT} in place 
  of \cite[Theorem~2.14]{BB-cohMod}, and with 
  \cite[Theorem~3.3.1]{EvansHrushovski} replacing \cite[Proposition 
  A.4]{BB-cohMod}. We describe the necessary changes.

  Firstly, \cite[Theorem~2.15]{BB-cohMod} goes through in the case that $X_i 
  \subseteq  ((K_r)^\U)^{n_i}$ for some $r$ ($i=1,2$). The proof is identical, using 
  Theorem~\ref{t:SzT}; the sublinearity of the dependence on $s$ where $K_{2,s}$ is 
  omitted, discussed after \cite[Theorem~2.14]{BB-cohMod}, also holds here: 
  this is described in \cite[Remark~2.7(2),Corollary~2.8]{CS-ES1d}, and is 
  proven explicitly in \cite[Theorem~2.6]{CPS-ES}. (In fact this sublinearity 
  isn't necessary for the present 1-dimensional case.)

  Now \cite[Theorem~5.9]{BB-cohMod} goes through for $P \subseteq  (K')^{<\omega}$. 
  The proof is identical, except that in the proof of 
  \cite[Proposition~5.14]{BB-cohMod}, since $\a,\d \in (K')^{<\omega}$, already 
  $\a,\d \in ((K_r)^\U)^{<\omega}$ for some $r$, and this passes through to the 
  types $X_i$ since $(K_r)^\U$ is definable, so the above restricted form of 
  \cite[Theorem~2.15]{BB-cohMod} applies.

  Next, $\End^0_{C_0}(G)$ must be redefined as the skew-field of quotients of 
  $\End_{C_0}(G)$ (this agrees with $\Q\otimes\End_{C_0}(G)$ in characteristic 
  0); see \cite[3.1]{EvansHrushovski} for discussion of the possibilities.

\newcommand{\G}{{\mathcal{G}}}
\newcommand{\x}{{\overline{x}}}
\newcommand{\h}{{\overline{h}}}
  Finally, we indicate how to circumvent the use of \cite[Proposition 
  A.4]{BB-cohMod}, which is proven only in characteristic 0, in the 
  1-dimensional case. Where this is applied in 
  \cite[Proposition~6.1]{BB-cohMod},
  we have $a_i \in K'$ ($i=1,\ldots ,n$) such that $\G_\a = \{ \acl^0(a_i) : i \}$ 
  embeds in a projective subgeometry of the $\acl^0$-geometry $\G_K$ of $K$. 
  (Here we have $a_i \in K'$ rather than $a_i \in K^{<\omega}$, as this is what 
  arises in the proof, via \cite[Theorem~7.4]{BB-cohMod}, in the 1-dimensional 
  case corresponding to the statement of the current theorem.)
  By \cite[Theorem~3.3.1]{EvansHrushovski}, there is a 1-dimensional algebraic 
  group $G$ over $C_0$ and generic $\x \in G^m$ over $C_0$ (where 
  $m=\dim(\G_\a)$) and $A \in \operatorname{Mat}_{n,m}(\End(G))$ such that, setting $\h := 
  A\x$, we have $\acl^0(h_i) = \acl^0(a_i)$. Then $\operatorname{loc}^0(\h) = AG^m$ is a 
  connected algebraic subgroup of $G^n$, as required.

  The rest of the proof goes through unchanged.
\end{proof}

\begin{remark}
  The only obstruction to pushing this to higher dimension, i.e.\ to a version 
  of \cite[Theorem~1.11]{BB-cohMod}, is the need to generalise the higher 
  dimensional version of Evans-Hrushovski \cite[Proposition A.4]{BB-cohMod} to 
  positive characteristic.

  Meanwhile, the proof of the converse direction (showing that every special 
  variety admits no powersaving) makes essential use of the characteristic 0 
  assumption in \cite[Proposition~7.10]{BB-cohMod}; this may not be so easy to 
  generalise, and the statement may need to change.

  For these reasons, we leave positive characteristic analogues of 
  \cite[Theorem~1.11]{BB-cohMod} to future work.
\end{remark}

\bibliographystyle{amsalpha}
\bibliography{finResST}

@article {ACGZ,
    AUTHOR = {Aschenbrenner, Matthias and Chernikov, Artem and Gehret, Allen
              and Ziegler, Martin},
     TITLE = {Distality in valued fields and related structures},
   JOURNAL = {Trans. Amer. Math. Soc.},
  FJOURNAL = {Transactions of the American Mathematical Society},
    VOLUME = {375},
      YEAR = {2022},
    NUMBER = {7},
     PAGES = {4641--4710},
      ISSN = {0002-9947},
   MRCLASS = {03C45 (03C60 12J25 12L12)},
  MRNUMBER = {4439488},
       DOI = {10.1090/tran/8661},
       URL = {https://doi.org/10.1090/tran/8661},
}

@article {CGS,
    AUTHOR = {Chernikov, Artem and Galvin, David and Starchenko, Sergei},
     TITLE = {Cutting lemma and {Z}arankiewicz's problem in distal
              structures},
   JOURNAL = {Selecta Math. (N.S.)},
  FJOURNAL = {Selecta Mathematica. New Series},
    VOLUME = {26},
      YEAR = {2020},
    NUMBER = {2},
     PAGES = {Paper No. 25, 27},
      ISSN = {1022-1824},
   MRCLASS = {03C64 (03C45 05C35 05D40)},
  MRNUMBER = {4079189},
       DOI = {10.1007/s00029-020-0551-2},
       URL = {https://doi.org/10.1007/s00029-020-0551-2},
}

@article{CS-ES1d,
    AUTHOR = {Chernikov, Artem and Starchenko, Sergei},
     TITLE = {Model-theoretic {E}lekes-{S}zab\'{o} in the strongly minimal case},
   JOURNAL = {J. Math. Log.},
  FJOURNAL = {Journal of Mathematical Logic},
    VOLUME = {21},
      YEAR = {2021},
    NUMBER = {2},
     PAGES = {Paper No. 2150004, 20},
      ISSN = {0219-0613},
   MRCLASS = {03C45},
  MRNUMBER = {4290493},
       DOI = {10.1142/S0219061321500045},
       URL = {https://doi.org/10.1142/S0219061321500045},
}

@article{BB-cohMod,
  AUTHOR = {Bays, Martin and Breuillard, Emmanuel},
  TITLE  = {Projective geometries arising from {E}lekes-{S}zabo problems},
   JOURNAL = {Ann. Sci. \'{E}c. Norm. Sup\'{e}r. (4)},
FJOURNAL = {Annales Scientifiques de l'\'{E}cole Normale Sup\'{e}rieure. Quatri\`eme
            S\'{e}rie},
  VOLUME = {54},
    YEAR = {2021},
  NUMBER = {3},
   PAGES = {627--681},
    ISSN = {0012-9593},
     DOI = {10.24033/asens.2467},
     URL = {https://doi.org/10.24033/asens.2467},
}

@article{CPS-ES,
  AUTHOR = {Chernikov, Artem and Peterzil, Ya'acov and Starchenko, Sergei},
  TITLE  = {Model-theoretic {E}lekes-{S}zab{\'o} for stable and o-minimal 
    hypergraphs},
  URL = {https://arxiv.org/abs/2104.02235v1},
  NOTE   = {arxiv:2104.02235v1},
}

@article {CS-extDefII,
    AUTHOR = {Chernikov, Artem and Simon, Pierre},
     TITLE = {Externally definable sets and dependent pairs {II}},
   JOURNAL = {Trans. Amer. Math. Soc.},
  FJOURNAL = {Transactions of the American Mathematical Society},
    VOLUME = {367},
      YEAR = {2015},
    NUMBER = {7},
     PAGES = {5217--5235},
      ISSN = {0002-9947},
   MRCLASS = {03C45 (03C50 68R05)},
  MRNUMBER = {3335415},
MRREVIEWER = {Alf Onshuus},
       DOI = {10.1090/S0002-9947-2015-06210-2},
       URL = {https://doi.org/10.1090/S0002-9947-2015-06210-2},
}

@article {simon-decomposition,
    AUTHOR = {Simon, Pierre},
     TITLE = {Type decomposition in {NIP} theories},
   JOURNAL = {J. Eur. Math. Soc. (JEMS)},
  FJOURNAL = {Journal of the European Mathematical Society (JEMS)},
    VOLUME = {22},
      YEAR = {2020},
    NUMBER = {2},
     PAGES = {455--476},
      ISSN = {1435-9855},
   MRCLASS = {03C45},
  MRNUMBER = {4049222},
       DOI = {10.4171/jems/926},
       URL = {https://doi.org/10.4171/jems/926},
}

@book{EnglerPrestel,
    AUTHOR = {Engler, Antonio J. and Prestel, Alexander},
     TITLE = {Valued fields},
    SERIES = {Springer Monographs in Mathematics},
 PUBLISHER = {Springer-Verlag},
   ADDRESS = {Berlin},
      YEAR = {2005},
     PAGES = {x+205},
      ISBN = {978-3-540-24221-5; 3-540-24221-X},
   MRCLASS = {12J20 (12F05 12J10 12J12 12J15)},
  MRNUMBER = {MR2183496 (2007a:12005)},
MRREVIEWER = {Niels Schwartz},
}

@book{Lang-algebra,
    AUTHOR = {Lang, Serge},
     TITLE = {Algebra},
    SERIES = {Graduate Texts in Mathematics},
    VOLUME = {211},
   EDITION = {third},
 PUBLISHER = {Springer-Verlag},
   ADDRESS = {New York},
      YEAR = {2002},
     PAGES = {xvi+914},
      ISBN = {0-387-95385-X},
   MRCLASS = {00A05 (15-02)},
  MRNUMBER = {MR1878556 (2003e:00003)},
}

@article {ES,
    AUTHOR = {Elekes, Gy\"orgy and Szab\'o, Endre},
     TITLE = {How to find groups? (and how to use them in {E}rd\"os
              geometry?)},
   JOURNAL = {Combinatorica},
  FJOURNAL = {Combinatorica. An International Journal on Combinatorics and
              the Theory of Computing},
    VOLUME = {32},
      YEAR = {2012},
    NUMBER = {5},
     PAGES = {537--571},
      ISSN = {0209-9683},
   MRCLASS = {05A16 (14N10)},
  MRNUMBER = {3004808},
MRREVIEWER = {Martin Klazar},
       DOI = {10.1007/s00493-012-2505-6},
       URL = {http://dx.doi.org/10.1007/s00493-012-2505-6},
}

@article {EvansHrushovski,
    AUTHOR = {Evans, David M. and Hrushovski, Ehud},
     TITLE = {Projective planes in algebraically closed fields},
   JOURNAL = {Proc. London Math. Soc. (3)},
  FJOURNAL = {Proceedings of the London Mathematical Society. Third Series},
    VOLUME = {62},
      YEAR = {1991},
    NUMBER = {1},
     PAGES = {1--24},
      ISSN = {0024-6115},
   MRCLASS = {05B25 (03C45 03C60 03D35 05B35 51E12)},
  MRNUMBER = {1078211},
MRREVIEWER = {Daniel Lascar},
       URL = {https://doi.org/10.1112/plms/s3-62.1.1},
}

@book {TaoVu,
    AUTHOR = {Tao, Terence and Vu, Van},
     TITLE = {Additive combinatorics},
    SERIES = {Cambridge Studies in Advanced Mathematics},
    VOLUME = {105},
 PUBLISHER = {Cambridge University Press, Cambridge},
      YEAR = {2006},
     PAGES = {xviii+512},
      ISBN = {978-0-521-85386-6; 0-521-85386-9},
   MRCLASS = {11-02 (05-02 05D10 11B13 11P70 11P82 28D05 37A45)},
  MRNUMBER = {2289012},
MRREVIEWER = {Serge\u\i  V. Konyagin},
       DOI = {10.1017/CBO9780511755149},
       URL = {https://doi.org/10.1017/CBO9780511755149},
}

@article {FoxEtAl,
    AUTHOR = {Fox, Jacob and Pach, J\'anos and Sheffer, Adam and Suk, Andrew
              and Zahl, Joshua},
     TITLE = {A semi-algebraic version of {Z}arankiewicz's problem},
   JOURNAL = {J. Eur. Math. Soc. (JEMS)},
  FJOURNAL = {Journal of the European Mathematical Society (JEMS)},
    VOLUME = {19},
      YEAR = {2017},
    NUMBER = {6},
     PAGES = {1785--1810},
      ISSN = {1435-9855},
   MRCLASS = {05C35 (05C25 52C10)},
  MRNUMBER = {3646875},
       DOI = {10.4171/JEMS/705},
       URL = {http://dx.doi.org/10.4171/JEMS/705},
}

@article {BKT,
    AUTHOR = {Bourgain, J. and Katz, N. and Tao, T.},
     TITLE = {A sum-product estimate in finite fields, and applications},
   JOURNAL = {Geom. Funct. Anal.},
  FJOURNAL = {Geometric and Functional Analysis},
    VOLUME = {14},
      YEAR = {2004},
    NUMBER = {1},
     PAGES = {27--57},
      ISSN = {1016-443X},
   MRCLASS = {11B75 (11T30)},
  MRNUMBER = {2053599},
MRREVIEWER = {Ben Joseph Green},
       DOI = {10.1007/s00039-004-0451-1},
       URL = {http://dx.doi.org/10.1007/s00039-004-0451-1},
}

@article {SZ,
     AUTHOR = {Stevens, Sophie and de Zeeuw, Frank},
      TITLE = {An improved point-line incidence bound over arbitrary fields},
    JOURNAL = {Bull. Lond. Math. Soc.},
   FJOURNAL = {Bulletin of the London Mathematical Society},
     VOLUME = {49},
       YEAR = {2017},
     NUMBER = {5},
      PAGES = {842--858},
       ISSN = {0024-6093},
    MRCLASS = {52C10 (14N20 51A05 52C35)},
   MRNUMBER = {3742451},
 MRREVIEWER = {Shira Zerbib},
        DOI = {10.1112/blms.12077},
        URL = {https://doi.org/10.1112/blms.12077},
}

@incollection {Weispfenning,
     AUTHOR = {Weispfenning, Volker},
      TITLE = {Quantifier elimination and decision procedures for valued
               fields},
  BOOKTITLE = {Models and sets ({A}achen, 1983)},
     SERIES = {Lecture Notes in Math.},
     VOLUME = {1103},
      PAGES = {419--472},
  PUBLISHER = {Springer, Berlin},
       YEAR = {1984},
    MRCLASS = {03C60 (03B25 03C10 12J10 12L05)},
   MRNUMBER = {775704},
 MRREVIEWER = {Philip Scowcroft},
        DOI = {10.1007/BFb0099397},
        URL = {https://doi.org/10.1007/BFb0099397},
}

@article {Hr-psfDims,
    AUTHOR = {Hrushovski, Ehud},
     TITLE = {On pseudo-finite dimensions},
   JOURNAL = {Notre Dame J. Form. Log.},
  FJOURNAL = {Notre Dame Journal of Formal Logic},
    VOLUME = {54},
      YEAR = {2013},
    NUMBER = {3-4},
     PAGES = {463--495},
      ISSN = {0029-4527},
   MRCLASS = {03C13 (03C45 05E15 11B30)},
  MRNUMBER = {3091666},
MRREVIEWER = {G. Cherlin},
       DOI = {10.1215/00294527-2143952},
       URL = {http://dx.doi.org/10.1215/00294527-2143952},
}

@article{KSW,
  title={Artin-{S}chreier extensions in {NIP} and simple fields},
  author={Kaplan, Itay and Scanlon, Thomas and Wagner, Frank O},
  journal={Israel Journal of Mathematics},
  volume={185},
  number={1},
  pages={141--153},
  year={2011},
  publisher={Springer}
}

@book {TZ,
    AUTHOR = {Tent, Katrin and Ziegler, Martin},
     TITLE = {A course in model theory},
    SERIES = {Lecture Notes in Logic},
    VOLUME = {40},
 PUBLISHER = {Association for Symbolic Logic, La Jolla, CA; Cambridge
              University Press, Cambridge},
      YEAR = {2012},
     PAGES = {x+248},
      ISBN = {978-0-521-76324-0},
   MRCLASS = {03C35 (03-01 03C45)},
  MRNUMBER = {2908005},
MRREVIEWER = {David Evans},
       URL = {https://doi.org/10.1017/CBO9781139015417},
}

@article {Holly,
    AUTHOR = {Holly, Jan E.},
     TITLE = {Canonical forms for definable subsets of algebraically closed
              and real closed valued fields},
   JOURNAL = {J. Symbolic Logic},
  FJOURNAL = {The Journal of Symbolic Logic},
    VOLUME = {60},
      YEAR = {1995},
    NUMBER = {3},
     PAGES = {843--860},
      ISSN = {0022-4812},
   MRCLASS = {03C60 (12L99)},
  MRNUMBER = {1348997},
MRREVIEWER = {M. Yasuhara},
       DOI = {10.2307/2275760},
       URL = {https://doi.org/10.2307/2275760},
}
\end{document}